\documentclass{amsart}
\usepackage[utf8]{inputenc}
\usepackage[T1]{fontenc}
\usepackage{amssymb,amsfonts,latexsym,amsmath,amscd,amsthm,mathtools}
\usepackage{tikz-cd,ytableau} 
\usepackage{epsfig,epic,eepic} 
\usepackage{pdfpages}
\usepackage{enumerate}
\usepackage{caption,subcaption}
\usepackage{cite}
\usepackage{hyperref}
\usepackage{listings,comment}

\usepackage[marginratio=1:1, textheight=608pt, textwidth=444pt]{geometry}

\newtheorem{definition}{Definition}[section]
\newtheorem{cor}[definition]{Corollary}
\newtheorem{lem}[definition]{Lemma}

\newtheorem{thm}[definition]{Theorem}

\newtheorem{remark}{Remark}[definition]

\numberwithin{equation}{section}

\newcommand{\ZZ}{\mathbb{Z}}
\newcommand{\NN}{\mathbb{N}}
\newcommand{\QQ}{\mathbb{Q}}

\newcommand{\RR}{\mathbb{R}}
\newcommand{\CC}{\mathbb{C}}
\newcommand{\abs}[1]{\left\lvert#1\right\rvert}

\newcommand{\floor}[1]{\left\lfloor#1\right\rfloor}
\newcommand{\ceil}[1]{\left\lceil#1\right\rceil}
\DeclareMathOperator{\ord}{ord}
\DeclareMathOperator{\Sym}{Sym}
\DeclareMathOperator{\std}{std}
\DeclareMathOperator{\Ad}{Ad}
\DeclareMathOperator{\GL}{GL}

\DeclareMathOperator{\Gal}{Gal}

\DeclareMathOperator{\He}{He}

\title[Decomposition of symmetric powers for $\mathrm{GL}(n)$]{Conjectural decomposition of symmetric powers of automorphic representations for $\mathrm{GL}(n)$}

\author{Kin Ming Tsang}
\address{Department of Mathematics, University of British Columbia}
\email{kmtsang@math.ubc.ca}

\subjclass[2020]{11F70}
\keywords{automorphic representation, symmetric powers}

\begin{document}

\begin{abstract}
Let $\pi$ be a cuspidal automorphic representation for $\mathrm{GL}(n)$ over a number field. We establish a conditional upper bound on the number of cuspidal isobaric summands in the symmetric $k$-th power lift of $\pi$, assuming that the symmetric $m$-th power lift of $\pi$ is automorphic and cuspidal for all $m \leq k-1$, along with other specified Langlands functoriality conjectures. For sufficiently large $k$, the resulting bound is independent of the specific value of $k$.
We further extend our study to cases in which the cuspidality assumptions on the symmetric power lifts are relaxed.
\end{abstract}

\maketitle

\section{Introduction}



Let $\pi$ be a cuspidal automorphic representation for $\GL(n)$ over a number field $F$. We associate to $\pi$ the following incomplete $L$-function
\begin{align*}
    L^{S}(s, \pi) = \prod_{v \notin S} \det \left( I_{n} - A_{v}(\pi)(\mathrm{N}v)^{-s} \right)^{-1}
\end{align*}
for $\Re(s) > 1$, where $S$ is the set containing all archimedean places and the places at which $\pi$ is ramified, $\mathrm{N}v$ is the norm of $v$, and $A_{v}(\pi)$ is the Langlands conjugacy class of $\pi$ at $v$, which can be represented by a diagonal matrix with entries $\alpha_{1}(v, \pi), \dots, \alpha_{n}(v, \pi) \in \CC$.

For $k \geq 1$, we consider the usual symmetric $k$-th power representation
\begin{align*}
    \Sym^{k} : \GL_{n}(\CC) \to \GL_{m}(\CC) ,
\end{align*}
where $m = \binom{n+k-1}{k}$.
We then consider the incomplete symmetric $k$-th power $L$-function associated to $\pi$, given by 
\begin{align*}
    L^{S}(s, \pi, \Sym^{k}) = \prod_{v \notin S} \det \left( I_{m} - \Sym^{k}(A_{v}(\pi))(\mathrm{N}v)^{-s} \right)^{-1}
\end{align*}
for $s$ in a suitable right half-plane.
The principle of functoriality~\cite{Langlands_functoriality_1979} predicts that there exists an automorphic representation $\Pi$ for $\GL(m)$ such that 
\begin{align*}
    L(s, \pi, \Sym^{k}) = L(s, \Pi) .
\end{align*}
In this case, we say that the symmetric $k$-th power of $\pi$ is automorphic and denote $\Pi$ by $\Sym^{k}(\pi)$.

Let $\mathcal{N}(\Pi)$ denote the number of cuspidal isobaric summands of an automorphic representation $\Pi$. We are interested in bounding $\mathcal{N}(\Sym^{k}(\pi))$.
For $n=2$, the cases $k=2,3$, and $4$ are well known. In particular, we have the sharp bounds (by the work of Ramakrishnan~\cite{Ramakrishnan_modularity_of_Rankin--Selberg_2000} and Kim-Shahidi~\cite{Kim_Shahidi_cuspidality_2002, Kim_Shahidi_sym3_2002})
\begin{align*}
    \mathcal{N}(\Sym^{2}(\pi)) &\leq 3, &
    \mathcal{N}(\Sym^{3}(\pi)) &\leq 3, \text{ and} & 
    \mathcal{N}(\Sym^{4}(\pi)) &\leq 5.
\end{align*}
It is worth noting that the bounds $\mathcal{N}(\Sym^{2}(\pi)) \leq 3$ and $\mathcal{N}(\Sym^{4}(\pi)) \leq 5$ are trivial. 

For higher symmetric power lifts, automorphy is not known. Assuming that the symmetric fifth power lift is automorphic, Ramakrishnan~\cite{Ramakrishnan_sym_power_cusp_form_2009} proved that it is cuspidal, provided that the symmetric square, cube, and fourth power lifts are cuspidal. 
Furthermore, Ramakrishnan showed that if $\Sym^{6}(\pi)$ is cuspidal, then $\Sym^{k}(\pi)$ is cuspidal for all $k \geq 6$ under the assumption that $\Sym^{j}(\pi)$ is automorphic for all $j \leq 2k$.

For the case $n=3$, Walji~\cite{Walji_conjectural_decomposition_2022} proved that if the symmetric $m$-th power lifts are cuspidal for $k-3 \leq m \leq k-1$, then $\mathcal{N}(\Sym^{k}(\pi)) \leq 3$ for $k \geq 7$, under the assumption that the symmetric $j$-th power lifts are automorphic for all $k-3 \leq j \leq k$ and that certain Rankin--Selberg products are automorphic.
For $n=4$, Walji also showed that if the symmetric $m$-th power lifts are cuspidal for $k-3 \leq m \leq k-1$, then $\mathcal{N}(\Sym^{k}(\pi)) \leq 4$ for $k \geq 39$ under the assumption that the symmetric $j$-th power lifts are automorphic for all $k-4 \leq j \leq k$ and that certain Rankin--Selberg products are automorphic.

In this article, we generalize the ideas of Ramakrishnan~\cite{Ramakrishnan_sym_power_cusp_form_2009} and Walji~\cite{Walji_conjectural_decomposition_2022} to the case $n \geq 5$. More precisely, we bound the number of isobaric summands in the symmetric $k$-th power lift of $\pi$, where $\pi$ is a cuspidal automorphic representation for $\GL(n)$, for $n \geq 5$.
To state our theorems concisely, we define
\begin{align}\label{eqn:delta_nk_gamma_def}
    \delta_{n,k}(\gamma) = \ceil{ \frac{1}{n^{\gamma-1}} \min_{0 \leq i \leq \floor{\frac{n-1}{2}}} \ceil{\binom{n+k-\gamma-1-2i}{n-1} \bigg/ \binom{n}{2i+1}}} .
\end{align}

\begin{thm}\label{thm:Main_Theorem_intro}
Let $\pi$ be a cuspidal automorphic representation for $\GL(n)$. Fix $k \geq n+1$. Assume that $\Sym^{m} (\pi)$ is automorphic for $k-n \leq m \leq k$, and cuspidal for $m = k-2i-1$ for all $0 \leq i \leq \floor{\frac{n-1}{2}}$. Further assume that the exterior power lift $\Lambda^{j}(\pi)$ is automorphic for all $2 \leq j \leq n-1$, and that the Rankin--Selberg products of its dual with every summand of $\Sym^{k}(\pi)$ are automorphic. Then,
\begin{equation*}
    \mathcal{N}(\Sym^{k}(\pi)) \leq \floor{ \binom{n+k-1}{n-1} \bigg/ \delta_{n,k}(1) } .
\end{equation*}
\end{thm}

We can relax the cuspidality assumptions in the theorem while still obtaining nontrivial bounds. Specifically, we no longer assume that $\Sym^{m} (\pi)$ is cuspidal for $k-\gamma < m \leq k-1$, where $\gamma > 1$. We obtain the following bound for the number of isobaric summands in the symmetric $k$-th power lift of $\pi$, which depends on $\gamma$.

\begin{thm}\label{thm:Main_Theorem_var1_intro}
Let $\pi$ be a cuspidal automorphic representation for $\GL(n)$. Fix $k \geq n+1$ and let $\gamma \leq k-n$. Assume that $\Sym^{m} (\pi)$ is automorphic for $k-n-\gamma+1 \leq m \leq k$, cuspidal for $k-n-\gamma+ 1 \leq m \leq k-\gamma$. Further assume that the exterior power lift $\Lambda^{j}(\pi)$ is automorphic for all $2 \leq j \leq n-1$, and that the Rankin--Selberg products of its dual with every summand of $\Sym^{\ell}(\pi)$ are automorphic for all $k-\gamma + 1 \leq \ell \leq k$. Then,
\begin{equation*}
    \mathcal{N}(\Sym^{k}(\pi)) \leq \floor{ \binom{n+k-1}{n-1} \bigg/ \delta_{n,k}(\gamma) } .
\end{equation*}
\end{thm}

We prove the two theorems by establishing a lower bound for the degrees of the isobaric summands of $\Sym^{k}(\pi)$, given by $\delta_{n,k}(\gamma)$.
For $n=2$, $k=7$, and $\gamma = 2$, the lower bound for the degrees of the isobaric summands of $\Sym^{k}(\pi)$, given by $\delta_{n,k}(\gamma)$, is sharp. This will be discussed in Section~\ref{sec:consequences}.

\subsection*{Outline of the article}
In Section 2, we cover the necessary background and prove a key Schur polynomial identity~\eqref{eqn:Schur_poly_fund_identity}. In Section 3, we prove Theorems~\ref{thm:Main_Theorem_intro} and~\ref{thm:Main_Theorem_var1_intro} and discuss some asymptotics of our bounds. In Section 4, we prove an auxiliary theorem regarding the symmetric square lift. In Section 5, we illustrate our bounds with some examples.

\section*{Acknowledgements}
The author would like to thank Nahid Walji and Michael Bennett for their supervision and helpful discussions.
The author was supported in part by an NSERC Discovery Grant.

\section{Background}

\subsection{Automorphic $L$-functions}
Let $F$ be a number field and $\mathbb{A}_{F}$ be its ad\`ele ring.
Let $\pi$ be an irreducible, automorphic representation for $\GL_{n}(\mathbb{A}_{F})$. The local $L$-function $L_{v}(s, \pi)$ at a non-archimedean place $v$ of $F$ is given by
\begin{align*}
    L_{v}(s,\pi) &= \det \left( I_{n} - A_{v}(\pi)(\mathrm{N}v)^{-s} \right)^{-1} \\
    &= \det \left( I_{n} - 
    \begin{pmatrix}
    \alpha_{1}(v, \pi) &  & & \\
     & \alpha_{2}(v, \pi) & & \\
    & & \ddots & \\
    & & & \alpha_{n}(v, \pi)
    \end{pmatrix}
    (\mathrm{N}v)^{-s} \right)^{-1}
\end{align*}
for $\Re(s) > 1$, where $\alpha_{1}(v, \pi), \dots, \alpha_{n}(v, \pi)$ are the Satake parameters. If $\pi$ is unramified at $v$, then $\alpha_{i}(v, \pi) \neq 0$ for all $i$. The standard $L$-function associated to $\pi$ is 
\begin{align*}
    L(s, \pi) = \prod_{v} L_{v}(s, \pi),
\end{align*}
where $v$ runs over all the places of $F$.
It is known that $L(s, \pi)$ has a simple pole at $s=1$ if and only if $\pi$ is equivalent to the trivial Hecke character. 

In this article, it suffices to consider the incomplete $L$-function associated to $\pi$, given by
\begin{align*}
    L^{T}(s, \pi) = \prod_{v \notin T} L_{v}(s, \pi),
\end{align*}
where $T$ is a finite set containing all archimedean places and the places at which $\pi$ is ramified.

\subsection{Isobaric representations}
Given cuspidal automorphic representations $\pi_{i}$ for $\GL(n_i)$ for $1 \leq i \leq k$ over a number field $F$, there exists an irreducible automorphic representation $\Pi$ for $\GL(n)$ over $F$ such that
\begin{align*}
    L(s,\Pi) = \prod_{j=1}^{k} L(s,\pi_{j}) ,
\end{align*}
where $n = n_{1} + n_{2} + \cdots + n_{k}$~\cite{Jacquet_Shalika_On_Euler_products_1981}. Each $\pi_{j}$ is called an isobaric summand of $\Pi$. 
We call $\Pi$ an isobaric automorphic representation and denote it by $\pi_{1} \boxplus \cdots \boxplus \pi_{k}$.
Let
\begin{align*}
    \mathcal{A}(\GL_{n}(\mathbb{A}_{F})) = \{ \text{equivalence classes of (unitary) isobaric automorphic representations for } \GL_{n}(\mathbb{A}_{F}) \}
\end{align*}
and let $\mathcal{A}_{0}(\GL_{n}(\mathbb{A}_{F})) \subset \mathcal{A}(\GL_{n}(\mathbb{A}_{F}))$ be the subset consisting of cuspidal automorphic representations.

We denote by $\mathcal{N}(\Pi)$ the number of isobaric summands of $\Pi$.
It is worth noting that $\mathcal{N}(\Pi) = 1$ if and only if $\Pi$ is cuspidal. 
We employ the same notation as in~\cite{Walji_conjectural_decomposition_2022} that if $\pi_{j}$ is an isobaric summand of $\Pi$, we write $\pi_{j} \prec \Pi$.

\subsection{Symmetric powers and exterior powers $L$-functions}
Let $\pi \in \mathcal{A}(\GL_{n}(\mathbb{A}_{F}))$.
Consider the usual symmetric $k$-th power representation $\Sym^{k}: \GL_{n}(\CC) \to \GL_{m}(\CC)$. We associate to $\pi$ its (incomplete) symmetric $k$-th power $L$-function 
\begin{align*}
    L^{T}(s,\pi, \Sym^{k}) &= \prod_{v \notin T} L_{v}(s,\pi, \Sym^{k}) \\
    &= \prod_{v \notin T} \det \left( I_{m} - \Sym^{k}(A_{v}(\pi)) (\mathrm{N}v)^{-s} \right)^{-1} ,
\end{align*}
where $\Sym^{k}$ of the conjugacy class $A_{v}(\pi)$ can be represented as a diagonal matrix whose entries are the monomials occurring in a complete homogeneous symmetric polynomial in $X_{i} = \alpha_{i}(v,\pi)$.
The diagonal entries are precisely the elements of the set
\begin{align*}
    \left\{ \prod_{i=1}^{n} \alpha_{i}(v,\pi)^{k_{i}} \mid \sum_{i=1}^{n} k_{i} = k, \text{ where each } k_{i} \geq 0 \right\} .
\end{align*}
Langlands functoriality~\cite{Langlands_functoriality_1979} predicts that there exists an isobaric automorphic representation $\Pi \in \mathcal{A}(\GL_{m}(\mathbb{A}_{F}))$ such that $L(s, \Pi) = L(s, \pi, \Sym^{k})$.
In this case, we say that the symmetric $k$-th power of $\pi$ is automorphic, and we denote $\Pi$ by $\Sym^{k} \pi$. For $n=2$, the symmetric square lift~\cite{Gelbart_Jacquet_sym2_1978}, symmetric cube lift~\cite{Kim_Shahidi_sym3_2002}, and symmetric fourth power lift~\cite{Kim_Sym4_2003} of $\pi$ are known to be automorphic. For $n \geq 3$, the automorphy of symmetric power lifts remains open in general.

One can define the $k$-th exterior power $L$-function in a similar fashion. Let $\Lambda^{k}: \GL_{n}(\CC) \to \GL_{m^{\prime}}(\CC)$ be the exterior $k$-th power representation, where $m^{\prime} = \binom{n}{k}$. We associate to $\pi$ the following (incomplete) exterior $k$-th power $L$-function 
\begin{align*}
    L^{T}(s,\pi, \Lambda^{k}) &= \prod_{v \notin T} L_{v}(s,\pi, \Lambda^{k}) \\
    &=\prod_{v \notin T} \det \left( I_{m^{\prime}} - \Lambda^{k}(A_{v}(\pi)) (\mathrm{N}v)^{-s} \right)^{-1} ,
\end{align*}
where $\Lambda^{k}$ of the matrix $A_{v}(\pi)$ can be represented as a diagonal matrix whose entries are the monomials occurring in an elementary symmetric polynomial in $X_{i} = \alpha_{i}(v,\pi)$.
The diagonal entries are precisely the elements of the set
\begin{align*}
    \left\{ \prod_{j=1}^{k} \alpha_{i_{j}}(v, \pi) \mid 1 \leq i_{1} < i_{2} < \cdots < i_{k} \leq n \right\} .
\end{align*}
Analogously, we say that the exterior $k$-th power of $\pi$ is automorphic if there exists $\Pi^{\prime} \in \mathcal{A}(\GL_{m^{\prime}}(\mathbb{A}_{F}))$ such that $L(s,\Pi^{\prime}) = L(s,\pi, \Lambda^{k})$. In this case, we denote $\Pi^{\prime}$ by $\Lambda^{k} \pi$.
When $n=3$, both the exterior square and cube lifts are automorphic. When $n=4$, the exterior square lift~\cite{Kim_Sym4_2003}, the exterior cube lift, and the exterior fourth power lift are all automorphic.

\subsection{Rankin--Selberg $L$-functions}
Let $\pi \in \mathcal{A}(\GL_{n}(\mathbb{A}_{F}))$ and $\pi^{\prime} \in \mathcal{A}(\GL_{m}(\mathbb{A}_{F}))$. We define the following (incomplete) Rankin--Selberg $L$-function associated to $\pi$ and $\pi^{\prime}$ by 
\begin{align*}
    L^{T}(s, \pi \times \pi^{\prime}) &= \prod_{v \notin T} L_{v}(s, \pi \times \pi^{\prime}) \\
    &= \prod_{v \notin T} \det \left( I_{nm} - (A_{v}(\pi) \otimes A_{v}(\pi^{\prime})) (\mathrm{N}v)^{-s} \right)^{-1} ,
\end{align*}
where $A_{v}(\pi) \otimes A_{v}(\pi^{\prime})$ can be represented as a diagonal matrix whose entries are precisely the elements of the set
\begin{align*}
    \left\{ \alpha_{i}(v, \pi) \alpha_{j}(v, \pi^{\prime}) \mid 1 \leq i \leq n \text{ and } 1 \leq j \leq m \right\} .
\end{align*}
Suppose that $\pi$ and $\pi'$ are cuspidal. Jacquet and Shalika~\cite{Jacquet_Shalika_On_Euler_products_1981} showed that $L^{S}(s, \pi \times \pi^{\prime})$ is entire unless $\pi^{\prime} \simeq \widetilde{\pi}$, where $\widetilde{\pi}$ is the dual of $\pi$. In this case, $L^{S}(s, \pi \times \pi^{\prime})$ has a simple pole at $s = 1$.

\subsection{Schur polynomials} 
We refer the reader to~\cite[Appendix A]{Fulton_Harris_representation} for a detailed account of Schur polynomials.
Let $\lambda = (\lambda_{1}, \dots, \lambda_{n})$ be a partition of $d$ with $\lambda_{1} \geq \cdots \geq \lambda_{n}$. The Schur polynomial associated to $\lambda$ is defined by
\begin{align*}
    S_{\lambda} = \frac{\det (x_{j}^{\lambda_{i}+n-i})}{\det (x_{j}^{n-i})} = \frac{\det (x_{j}^{\lambda_{i}+n-i})}{\Delta} ,
\end{align*}
where $\Delta = \prod_{i < j} (x_{i}-x_{j})$ denotes the Vandermonde determinant.


The following identity is known as the Littlewood--Richardson rule, which describes the multiplication of two Schur polynomials.
\begin{lem}
Let $\lambda$ be a partition of $d$ and $\mu$ a partition of $m$. Then
\begin{equation*}
    S_{\lambda} \cdot S_{\mu} = \sum_{\nu} N_{\lambda \mu \nu} S_{\nu} ,
\end{equation*}
where the sum is over all partitions $\nu$ of $d + m$ (each with at most $n$ parts) and
$N_{\lambda \mu \nu}$ is the number of ways the Young diagram for $\lambda$ can be expanded to that of $\nu$ by a strict $\mu$-expansion.
\end{lem}

Another related concept is the monomial symmetric function. It also forms a basis for the degree $d$ homogeneous symmetric polynomials in $n$ variables.
The monomial symmetric function associated to the partition $\lambda = (\lambda_{1}, \dots, \lambda_{n})$ is defined as
\begin{align*}
    M_{\lambda} = \sum_{\alpha} x_{1}^{\alpha_{1}} \cdots x_{n}^{\alpha_{n}} ,
\end{align*}
where the sum is over all permutations $\alpha = (\alpha_{1}, \dots, \alpha_{n})$ of $(\lambda_{1}, \dots, \lambda_{n})$.

To relate monomial symmetric functions to Schur polynomials, we introduce the Kostka numbers.
Let $\mu = (\mu_{1}, \dots, \mu_{m})$ and $\lambda = (\lambda_{1}, \dots, \lambda_{n})$ be partitions of $d$. The Kostka number $K_{\mu \lambda}$ is defined as the number of different semi-standard Young tableaux of shape $\mu$ and weight $\lambda$. Equivalently, it is the number of ways to fill the boxes of the Young diagram of $\mu$ with $\lambda_{1}$ entries of $1$, $\lambda_{2}$ entries of $2$, and so on, such that the entries are non-decreasing along each row and strictly increasing down each column.

The following lemma describes the relation between Schur polynomials, monomial symmetric functions, and Kostka numbers.
\begin{lem}
We have
\begin{equation*}
    S_{\mu} = \sum_{\lambda} K_{\mu \lambda} M_{\lambda} .
\end{equation*}
\end{lem}
\begin{proof}
See equation (A.19) in~\cite[p.458]{Fulton_Harris_representation}.
\end{proof}

We now record a fundamental identity for Schur polynomials that will play a key role in subsequent proofs.
\begin{lem}\label{lem:Schur_poly_fund_identity}
We have the following Schur polynomial identity. Note that all Schur polynomials involved are in $n$ variables.
\begin{align}\label{eqn:Schur_poly_fund_identity}
    \sum_{i=0}^{\floor{\frac{n}{2}}} S_{(k-2i)} \cdot S_{(\underbrace{\scriptstyle 1,1,\dots,1}_{2i})}
    = \sum_{i=0}^{\ceil{\frac{n}{2}}-1} S_{(k-2i-1)} \cdot S_{(\underbrace{\scriptstyle 1,1,\dots,1}_{2i+1})} .
\end{align}
\end{lem}
\begin{proof}
From the Littlewood--Richardson rule, we have the following identities:
\begin{align*}
    S_{(k-1)} \cdot S_{(1)} &= S_{(k)} + S_{(k-1,1)}, \\
    S_{(k-2)} \cdot S_{(1,1)} &= S_{(k-1,1)} + S_{(k-2,1,1)}, \\
    &\vdots \\
    S_{(k-i)} \cdot S_{(\underbrace{\scriptstyle 1,1,\dots,1}_{i \text{ times}})} &= S_{(k-i+1,\underbrace{\scriptstyle 1,1,\dots,1}_{i-1 \text{ times}})} + S_{(k-i,\underbrace{\scriptstyle 1,1,\dots,1}_{i \text{ times} })}, \\
    &\vdots \\
    S_{(k-n)} \cdot S_{(\underbrace{\scriptstyle 1,1,\dots,1}_{n \text{ times}})} &= S_{(k-n+1,\underbrace{\scriptstyle 1,1,\dots,1}_{n-1 \text{ times}})} .
\end{align*}
The general formula is true for all $0 \leq i \leq n$ given that we have the trivial identification 
\begin{equation*}
    S_{(m+1,\underbrace{\scriptstyle 1,1,\dots,1}_{n-1 \text{ times}})} = 0 \quad \text{ and } \quad S_{(m-n,\underbrace{\scriptstyle 1,1,\dots,1}_{n \text{ times} })} = 0 .
\end{equation*}
Consider the following two equations
\begin{align*}
    S_{(k-i)} \cdot S_{(\underbrace{\scriptstyle 1,1,\dots,1}_{i})} &= S_{(k-i+1,\underbrace{\scriptstyle 1,1,\dots,1}_{i-1})} + S_{(k-i,\underbrace{\scriptstyle 1,1,\dots,1}_{i})} ,\\
    S_{(k-i-1)} \cdot S_{(\underbrace{\scriptstyle 1,1,\dots,1}_{i+1})} &= S_{(k-i,\underbrace{\scriptstyle 1,1,\dots,1}_{i})} + S_{(k-i-1,\underbrace{\scriptstyle 1,1,\dots,1}_{i+1})} .
\end{align*}
Add the left-hand side of the first equation to the right-hand side of the second equation, we obtain
\begin{align*}
    S_{(k-i-1,\underbrace{\scriptstyle 1,1,\dots,1}_{i+1})}
    + 
    S_{(k-i)} \cdot S_{(\underbrace{\scriptstyle 1,1,\dots,1}_{i})} 
    = 
    S_{(k-i-1)} \cdot S_{(\underbrace{\scriptstyle 1,1,\dots,1}_{i+1})}
    +
    S_{(k-i+1,\underbrace{\scriptstyle 1,1,\dots,1}_{i-1})} .
\end{align*}
Lastly, we sum over all even integers $i = 0,2,4, \dots, 2 \floor{n/2}$.
\end{proof}

Let $\pi \in \mathcal{A}_{0}(\GL_{n}(\mathbb{A}_{F}))$. For any $v \notin T$, where $T$ is the set containing all archimedean places and the places at which $\pi$ is ramified, recall that $A_{v}(\pi)$ denotes the Langlands conjugacy class of $\pi$ at $v$, with eigenvalues $\alpha_{1}(v,\pi), \dots, \alpha_{n}(v,\pi)$.
We observe that the multiset of monomials, counted with multiplicity, appearing in $S_{(k)}(\alpha_{1}(v,\pi), \dots, \alpha_{n}(v,\pi))$ coincides with the set of eigenvalues of $\Sym^{k}(A_{v}(\pi))$.
Thus, there is a natural correspondence between the representation $\Sym^{k}(\pi)$ and the Schur polynomial $S_{(k)}$.
Similarly, we have a correspondence between the representation $\Lambda^{k}(\pi)$ and the Schur polynomial $S_{(\underbrace{\scriptstyle 1,1,\dots,1}_{k})}$. 
We conclude this section by presenting the correspondence table below.
\begin{center}
\begin{tabular}{c|c}
    \text{Representation} & \text{Schur polynomial} \\ [1mm]
    \hline
    $\Sym^{k}(\pi)$ & $S_{(k)}$ \\ [2mm]
    $\Lambda^{k}(\pi)$ & $S_{(\underbrace{\scriptstyle 1,1,\dots,1}_{k})}$ \\ [2mm]
    $\omega_{\pi}$ & $S_{(\underbrace{\scriptstyle 1,1,\dots,1}_{n})}$ \\ [2mm]
    $\omega_{\pi} \times \widetilde{\pi}$ & $S_{(\underbrace{\scriptstyle 1,1,\dots,1}_{n-1})}$ \\ [2mm]
    $\omega_{\pi}^{\otimes k} \times \widetilde{\Sym^{k}}(\pi)$ & $S_{(\underbrace{\scriptstyle k,k,\dots,k}_{n-1})}$ \\ [2mm]
    $\omega_{\pi} \times \widetilde{\Lambda^{k}}(\pi)$ & $S_{(\underbrace{\scriptstyle 1,1,\dots,1}_{n-k})}$ \\ [2mm]
    $\omega_{\pi} \times \Ad(\pi)$ & $S_{(2,\underbrace{\scriptstyle 1,1,\dots,1}_{n-2})} $
\end{tabular}
\end{center}
The correspondences can be established through a careful analysis of Schur polynomials via their combinatorial formula involving Kostka numbers.

\section{Main Theorem}

\begin{lem}\label{lem:Main_Lemma}
Let $\pi \in \mathcal{A}_{0}(\GL_{n}(\mathbb{A}_{F}))$. Let $\Sym^{0}$ denote the trivial representation and $\Sym^{1}$ denote the standard representation of $\GL_{n}(\mathbb{C})$. Let $T$ be a finite set consisting of all the finite places at which $\pi$ is ramified and all the archimedean places. For any integer $k \geq n$, we have, for $s$ in some suitable right half-plane,
\begin{align*}
    \prod_{i=0}^{\floor{\frac{n}{2}}} L^{T}(s,\pi,\Sym^{k-2i} \times \Lambda^{2i}) = \prod_{i=0}^{\ceil{\frac{n}{2}}-1} L^{T}(s,\pi,\Sym^{k-2i-1} \times \Lambda^{2i+1}) 
\end{align*}
or, equivalently,
\begin{align*}
    &L^{T}(s,\pi,\Sym^{k}) \prod_{i=1}^{\floor{\frac{n}{2}}} L^{T}(s,\pi,\Sym^{k-2i} \times \Lambda^{2i}) \\
    =& L^{T}(s,\pi,\Sym^{k-1} \times \std) \prod_{i=1}^{\ceil{\frac{n}{2}}-1} L^{T}(s,\pi,\Sym^{k-2i-1} \times \Lambda^{2i+1}) .
\end{align*}
\end{lem}
\begin{proof}
This is a direct application of Lemma~\ref{lem:Schur_poly_fund_identity} together with the correspondence table.
\end{proof}

\begin{thm}\label{thm:Main_Theorem}
Let $\pi \in \mathcal{A}_{0}(\GL_{n}(\mathbb{A}_{F}))$ be a cuspidal automorphic representation. Fix an integer $k \geq n+1$. Assume that $\Sym^{m} (\pi)$ is automorphic for $k-n \leq m \leq k$, and cuspidal for $m = k-2i-1$ for all $0 \leq i \leq \floor{\frac{n-1}{2}}$. Further assume that the exterior power lift $\Lambda^{j}(\pi)$ is automorphic for all $2 \leq j \leq n-1$, and that the Rankin--Selberg products of its dual with every summand of $\Sym^{k}(\pi)$ are automorphic. Then,
\begin{equation*}
    \mathcal{N}(\Sym^{k}(\pi)) \leq \floor{ \binom{n+k-1}{n-1} \bigg/ \delta_{n,k}(1)}
\end{equation*}
where $\delta_{n,k}(1)$ is defined in~\eqref{eqn:delta_nk_gamma_def}.
\end{thm}

\begin{proof}
Assume that $\Sym^{k}(\pi)$ is not cuspidal. Then it has a cuspidal summand $\tau$ for some $\GL(r)$ such that
\begin{equation*}
    r \leq \floor{\frac{1}{2}\binom{n+k-1}{n-1}} .
\end{equation*}
Then,
\begin{equation*}
    -\ord_{s=1} L^{T}(s,\Sym^{k}(\pi) \times \widetilde{\tau}) \geq 1 .
\end{equation*}
The automorphy assumptions on the Rankin--Selberg products of $\Lambda^{j}(\pi)$ with the dual of any summand of $\Sym^{k}(\pi)$, for any $2 \leq j \leq n-1$, imply that
\begin{align*}
    -\ord_{s=1} L^{T}(s,\Sym^{k-2i}(\pi) \times \Lambda^{2i}(\pi) \boxtimes \widetilde{\tau}) \geq 0 ,
\end{align*}
owing to the non-vanishing of Rankin--Selberg $L$-functions at $s=1$.
It follows from Lemma~\ref{lem:Main_Lemma}, which applies under our automorphy assumptions, that
\begin{align}\label{eqn:sum_L_functions_1}
    \sum_{i=0}^{\floor{\frac{n-1}{2}}} -\ord_{s=1} L^{T}(s,\Sym^{k-1-2i}(\pi) \times \Lambda^{2i+1}(\pi) \boxtimes \widetilde{\tau}) \geq 1 .
\end{align}
Hence, at least one of the $L$-functions in~\eqref{eqn:sum_L_functions_1} must have a pole at $s=1$.
Suppose that
\begin{align*}
    -\ord_{s=1} L^{T}(s,\Sym^{k-1-2i}(\pi) \times \Lambda^{2i+1}(\pi) \boxtimes \widetilde{\tau}) \geq 1.
\end{align*}
Then, under the cuspidality assumptions, it follows that
\begin{equation*}
    \Sym^{k-1-2i}(\pi) \prec \widetilde{\Lambda^{2i+1}(\pi)} \times \tau
\end{equation*}
and hence
\begin{equation*}
    \binom{n+k-2-2i}{n-1} \leq \binom{n}{2i+1} \, r .
\end{equation*}
A simple calculation gives
\begin{equation*}
    r \geq \ceil{\frac{\binom{n+k-2-2i}{n-1}}{\binom{n}{2i+1}}} .
\end{equation*}
Since we do not know exactly which $L$-function in~\eqref{eqn:sum_L_functions_1} has a pole at $s=1$, we can only conclude that
\begin{equation*}
    r \geq \min_{0 \leq i \leq \floor{\frac{n-1}{2}}} \ceil{\frac{\binom{n+k-2-2i}{n-1}}{\binom{n}{2i+1}}} = \delta_{n,k}(1).
\end{equation*}
The desired upper bound on $\mathcal{N}(\Sym^{k}(\pi))$ now follows.
\end{proof}

\begin{remark}\label{rmk:min_location_for_main_thm}
Let $k \geq n^{2+\varepsilon}$ for some $\varepsilon > 0$.
We show that the minimum of
\begin{align}\label{eqn:expression_of_delta_nk1}
    \ceil{\frac{\binom{n+k-2-2i}{n-1}}{\binom{n}{2i+1}}}
\end{align}
over $0 \le i \le \floor{\frac{n-1}{2}}$ is attained at $i=\floor{\frac{n}{4}}$. Since the cases $n = 2$ and $n = 3,4$ were treated by Ramakrishnan~\cite{Ramakrishnan_sym_power_cusp_form_2009} and Walji~\cite{Walji_conjectural_decomposition_2022}, respectively, we henceforth assume that $n \geq 5$. 
Note that $\binom{n+k-2-2i}{n-1}$ is decreasing as a function of $i$. Moreover, $\binom{n}{2i+1}$ increases for $0 \leq i \leq \floor{\frac{n}{4}}$ and decreases for $\floor{\frac{n}{4}} \leq i \leq \floor{\frac{n-1}{2}}$. Hence, it suffices to restrict our attention to the range $\floor{\frac{n}{4}} \leq i \leq \floor{\frac{n-1}{2}}$.

Consider the ratio
\begin{align}\label{eqn:ratio_for_remark_3.2.1}
\begin{aligned}
    \frac{\binom{n+k-2-2(i+1)}{n-1}}{\binom{n}{2(i+1)+1}} \bigg/ \frac{\binom{n+k-2-2i}{n-1}}{\binom{n}{2i+1}} 
    &= \frac{\binom{n+k-4-2i}{n-1}}{\binom{n+k-2-2i}{n-1}} \cdot \frac{\binom{n}{2i+1}}{\binom{n}{2i+3}} \\
    &= \frac{(k-1-2i)(k-2-2i)}{(n+k-2-2i)(n+k-3-2i)} \cdot \frac{(2i+3)(2i+2)}{(n-2i-1)(n-2i-2)} ,
\end{aligned}
\end{align}
where $\floor{\frac{n}{4}} \leq i \leq \floor{\frac{n-1}{2}}-1$.
The first factor $\frac{(k-1-2i)(k-2-2i)}{(n+k-2-2i)(n+k-3-2i)}$ is a decreasing function of $i$. This follows from the observation that $\frac{a-1}{b-1} < \frac{a}{b}$ for $0< a < b$.
Hence,
\begin{align*}
    \frac{(k-1-2i)(k-2-2i)}{(n+k-2-2i)(n+k-3-2i)} &\geq \frac{(k-1-2(\frac{n-1}{2}-1))(k-2-2(\frac{n-1}{2}-1))}{(n+k-2-2(\frac{n-1}{2}-1))(n+k-3-2(\frac{n-1}{2}-1))} \\
    &= \frac{(k-n+2)(k-n+1)}{(k+1)k} \\
    &> \frac{(k-n+1)^{2}}{k^{2}}.
\end{align*}
Next, we seek a lower bound for $\frac{(2i+3)(2i+2)}{(n-2i-1)(n-2i-2)}$. For fixed $n \geq 5$, this expression is an increasing function of $i$. Hence, the minimum is attained at $i = \floor{\frac{n}{4}}$.
Observe that $n - 2i = n - 2\floor{\frac{n}{4}} \leq n - \frac{2n}{4} = \frac{n}{2}$, and that $2i = 2\floor{\frac{n}{4}} \geq 2(\frac{n}{4}-1) = \frac{n}{2}-2$.
Then,
\begin{align}\label{eqn:binomial_ratio_bounds}
    \frac{(2i+3)(2i+2)}{(n-2i-1)(n-2i-2)} \geq \frac{(\frac{n}{2}+1)\frac{n}{2}}{(\frac{n}{2}-1)(\frac{n}{2}-2)} = \frac{(n+2)n}{(n-2)(n-4)} > \frac{(n+2)^{2}}{(n-2)^{2}}.
\end{align}
If the ratio in~\eqref{eqn:ratio_for_remark_3.2.1} is greater than $1$ for all $\floor{\frac{n}{4}} \leq i \leq \floor{\frac{n-1}{2}}-1$, then the quantity in~\eqref{eqn:expression_of_delta_nk1} attains its minimum at $i = \floor{\frac{n}{4}}$. It suffices to show that
\begin{align*}
    \left(\frac{k-n+1}{k}\right) \left( \frac{n+2}{n-2} \right)> 1 ,
\end{align*}
which is equivalent to showing that $4k > n^2+n-2$. This holds since $k \geq n^{2+\varepsilon} > n^{2}$.
\end{remark}

\begin{remark}\label{rmk:asymptotics_for_main_thm}
Let $k \geq n^{2+\varepsilon}$. We study the asymptotics of the bound
\begin{equation*}
    \mathcal{N}(\Sym^{k}(\pi)) \leq \floor{ \binom{n+k-1}{n-1} \bigg/ \ceil{\frac{\binom{n+k-2-2\floor{\frac{n}{4}}}{n-1}}{\binom{n}{2\floor{\frac{n}{4}}+1}}} } .
\end{equation*}
It suffices to analyze
\begin{equation*}
    \binom{n+k-1}{n-1} \bigg/ \ceil{\frac{\binom{n+k-2-2\floor{\frac{n}{4}}}{n-1}}{\binom{n}{2\floor{\frac{n}{4}}+1}}} .
\end{equation*}
Consider the ratio
\begin{align*}
    \frac{\binom{n+1+k-2-2\floor{\frac{n+1}{4}}}{n+1-1}}{\binom{n+1}{2\floor{\frac{n+1}{4}}+1}} \bigg/ \frac{\binom{n+k-2-2\floor{\frac{n}{4}}}{n-1}}{\binom{n}{2\floor{\frac{n}{4}}+1}} = \frac{\binom{n+k-1-2\floor{\frac{n+1}{4}}}{n}}{\binom{n+k-2-2\floor{\frac{n}{4}}}{n-1}} \cdot \frac{\binom{n}{2\floor{\frac{n}{4}}+1}}{\binom{n+1}{2\floor{\frac{n+1}{4}}+1}} .
\end{align*}
The first factor satisfies
\begin{align*}
    \frac{\binom{n+k-1-2\floor{\frac{n+1}{4}}}{n}}{\binom{n+k-2-2\floor{\frac{n}{4}}}{n-1}} 
    \geq \frac{\binom{n+k-1-2\floor{\frac{n}{4}}-2}{n}}{\binom{n+k-2-2\floor{\frac{n}{4}}}{n-1}} 
    = \frac{(k-2\floor{\frac{n}{4}}-1)(k-2\floor{\frac{n}{4}}-2)}{n(k+n-2\floor{\frac{n}{4}}-2)} .
\end{align*}
For the second factor, if $n \equiv 0,1,2 \pmod{4}$, then $\floor{\frac{n+1}{4}} = \floor{\frac{n}{4}}$, and 
\begin{align}\label{eqn:second_factor_analysis_012}
    \frac{\binom{n}{2\floor{\frac{n}{4}}+1}}{\binom{n+1}{2\floor{\frac{n+1}{4}}+1}} 
    = \frac{\binom{n}{2\floor{\frac{n}{4}}+1}}{\binom{n+1}{2\floor{\frac{n}{4}}+1}} 
    = \frac{n-2\floor{\frac{n}{4}}}{n+1} .
\end{align}
If $n \equiv 3 \pmod{4}$, then $\floor{\frac{n+1}{4}} = \floor{\frac{n}{4}}+1$ and $n = 4\floor{\frac{n}{4}}+3$. Hence, 
\begin{align}\label{eqn:second_factor_analysis_3}
    \frac{\binom{n}{2\floor{\frac{n}{4}}+1}}{\binom{n+1}{2\floor{\frac{n+1}{4}}+1}} 
    = \frac{\binom{n}{2\floor{\frac{n}{4}}+1}}{\binom{n+1}{2\floor{\frac{n}{4}}+3}} 
    = \frac{(2\floor{\frac{n}{4}}+3)(2\floor{\frac{n}{4}}+2)}{(n+1)(n-2\floor{\frac{n}{4}}-1)} 
    = \frac{n-2\floor{\frac{n}{4}}}{n+1} .
\end{align}
Therefore, the ratio 
\begin{align*}
    \frac{\binom{n+1+k-2-2\floor{\frac{n+1}{4}}}{n+1-1}}{\binom{n+1}{2\floor{\frac{n+1}{4}}+1}} \bigg/ \frac{\binom{n+k-2-2\floor{\frac{n}{4}}}{n-1}}{\binom{n}{2\floor{\frac{n}{4}}+1}} 
    \geq \frac{(k-2\floor{\frac{n}{4}}-1)(k-2\floor{\frac{n}{4}}-2)}{n(k+n-2\floor{\frac{n}{4}}-2)} \cdot \frac{n-2\floor{\frac{n}{4}}}{n+1}\gg \frac{k}{n} \to \infty
\end{align*}
as $n \to \infty$, since $k \geq n^{2+\varepsilon}$.

We have shown that 
\begin{align*}
    \frac{\binom{n+k-2-2\floor{\frac{n}{4}}}{n-1}}{\binom{n}{2\floor{\frac{n}{4}}+1}} \to \infty .
\end{align*}
Since $g(x) \to \infty$ implies $\frac{1}{\ceil{g(x)}} \sim \frac{1}{g(x)}$, we obtain 
\begin{align*}
    \binom{n+k-1}{n-1} \bigg/ \ceil{\frac{\binom{n+k-2-2\floor{\frac{n}{4}}}{n-1}}{\binom{n}{2\floor{\frac{n}{4}}+1}}}
    \sim \frac{\binom{n+k-1}{n-1}}{\binom{n+k-2-2\floor{\frac{n}{4}}}{n-1}} \cdot \binom{n}{2\floor{\frac{n}{4}}+1} .
\end{align*}
Moreover,
\begin{align*}
    \frac{\binom{n+k-1}{n-1}}{\binom{n+k-2-2\floor{\frac{n}{4}}}{n-1}}
    = \prod_{j=0}^{2\floor{\frac{n}{4}}} \frac{n+k-1-j}{k-j} 
    = \prod_{j=0}^{2\floor{\frac{n}{4}}} \left( 1+ \frac{n-1}{k-j}\right) 
    \leq \left( 1+ \frac{n-1}{n^{2+\varepsilon}-\frac{n}{2}}\right)^{\frac{n}{2}+1} ,
\end{align*}
which tends to $1$ as $n \to \infty$. Hence,
\begin{align*}
    \binom{n+k-1}{n-1} \bigg/ \ceil{\frac{\binom{n+k-2-2\floor{\frac{n}{4}}}{n-1}}{\binom{n}{2\floor{\frac{n}{4}}+1}}} \sim \binom{n}{2\floor{\frac{n}{4}}+1} .
\end{align*}
Note that the binomial coefficient $\binom{n}{2\floor{\frac{n}{4}}+1}$ is the central binomial coefficient except when $n \equiv 0 \pmod 4$, in which case
\begin{align*}
    \binom{n}{2\floor{\frac{n}{4}}+1} = \binom{n}{\frac{n}{2}+1} = \binom{n}{\frac{n}{2}} \cdot \frac{\frac{n}{2}}{\frac{n}{2} + 1} \sim \binom{n}{\frac{n}{2}} .
\end{align*}
By Stirling's formula (see e.g.~\cite[formula 6.1.37]{Abramowitz_Stegun_handbook_1964}), we conclude that
\begin{align*}
    \floor{\binom{n+k-1}{n-1} \bigg/ \ceil{\frac{\binom{n+k-2-2\floor{\frac{n}{4}}}{n-1}}{\binom{n}{2\floor{\frac{n}{4}}+1}}}} \sim \binom{n}{\floor{\frac{n}{2}}} \sim \frac{2^{n+\frac{1}{2}}}{\sqrt{\pi n}} ,
\end{align*}
which is independent of the specific value of $k$.
\end{remark}

In the following, we no longer assume that $\Sym^{m} (\pi)$ is cuspidal for $k-\gamma < m \leq k-1$, where $\gamma > 1$. Our goal is to obtain nontrivial bounds on the number of isobaric summands in the symmetric $k$-th power lift of $\pi$ under these weaker assumptions.
\begin{thm}\label{thm:Main_Theorem_var1}
Let $\pi \in \mathcal{A}_{0}(\GL_{n}(\mathbb{A}_{F}))$ be a cuspidal automorphic representation. Fix positive integers $k \geq n+1$ and $\gamma \leq k-n$. Assume that $\Sym^{m} (\pi)$ is automorphic for $k-n-\gamma+1 \leq m \leq k$, and cuspidal for $k-n-\gamma + 1 \leq m \leq k-\gamma$. Further assume that the exterior power lift $\Lambda^{j}(\pi)$ is automorphic for all $2 \leq j \leq n-1$, and that the Rankin--Selberg products of its dual with every summand of $\Sym^{\ell}(\pi)$ are automorphic for all $k-\gamma+1 \leq \ell \leq k$. Then,
\begin{equation*}
    \mathcal{N}(\Sym^{k}(\pi)) \leq \floor{ \binom{n+k-1}{n-1} \bigg/ \delta_{n,k}(\gamma) } ,
\end{equation*}
where $\delta_{n,k}(\gamma)$ is defined in~\eqref{eqn:delta_nk_gamma_def}.
\end{thm}
\begin{proof}
The idea is to follow the proof of Theorem~\ref{thm:Main_Theorem} to obtain the lower bound on the degree of every isobaric summand of $\Sym^{k-\gamma+\alpha}(\pi)$, where $1 \leq \alpha \leq \gamma$, inductively.
In particular, we will show that the degree $r_{\alpha}$ of every isobaric summand in $\Sym^{k-\gamma+\alpha}(\pi)$ is bounded below by $\delta_{n,k-\gamma+\alpha}(\alpha)$.

For the base case $\alpha = 1$, assume that $\Sym^{k-\gamma+1}(\pi)$ is not cuspidal. Then it has a cuspidal summand $\tau_{1}$ for some $\GL(r_{1})$. The proof of Theorem~\ref{thm:Main_Theorem} yields
\begin{equation*}
    r_{1} \geq \delta_{n,k-\gamma+1}(1) = \min_{0 \leq i \leq \floor{\frac{n-1}{2}}} \ceil{\frac{\binom{n+k-\gamma-1-2i}{n-1}}{\binom{n}{2i+1}}} .
\end{equation*}

Suppose that the degree $r_{\alpha}$ of every isobaric summand in $\Sym^{k-\gamma+\alpha}(\pi)$ satisfies $r_{\alpha} \geq \delta_{n,k-\gamma+\alpha}(\alpha)$ for all $1 \leq \alpha \leq \beta$, where $\beta < \gamma$ is a positive integer. We proceed by induction on $\alpha$.

We now show that the statement holds for $\alpha = \beta + 1$. Assume that $\Sym^{k-\gamma+\beta+1}(\pi)$ is not cuspidal. Then, it has a cuspidal summand $\tau_{\beta+1}$ for some $\GL(r_{\beta+1})$. 
Then,
\begin{equation*}
    -\ord_{s=1} L^{T}(s,\Sym^{k-\gamma+\beta+1}(\pi) \times \widetilde{\tau_{\beta+1}}) \geq 1 .
\end{equation*}
The automorphy assumptions on the Rankin--Selberg products of $\Lambda^{j}(\pi)$ with the dual of any summand of $\Sym^{k-\gamma+\beta+1}(\pi)$ for any $2 \leq j \leq n-1$ imply that
\begin{align*}
    -\ord_{s=1} L^{T}(s,\Sym^{k-\gamma+\beta+1-2i}(\pi) \times \Lambda^{2i}(\pi) \boxtimes \widetilde{\tau_{\beta+1}}) \geq 0 ,
\end{align*}
owing to non-vanishing of Rankin--Selberg $L$-functions at $s=1$.
It follows from Lemma~\ref{lem:Main_Lemma}, which applies under our automorphy assumptions, that
\begin{align}\label{eqn:sum_L_functions_general}
    \sum_{i=0}^{\floor{\frac{n-1}{2}}} -\ord_{s=1} L^{T}(s,\Sym^{k-\gamma+\beta-2i}(\pi) \times \Lambda^{2i+1}(\pi) \boxtimes \widetilde{\tau_{\beta+1}}) \geq 1 .
\end{align}
Hence, at least one of the $L$-functions in~\eqref{eqn:sum_L_functions_general} must have a pole at $s=1$. We consider three separate cases according to the possible ranges of $i$.

\underline{Case I:}
Suppose that 
\begin{align*}
    -\ord_{s=1} L^{T}(s,\Sym^{k-\gamma+\beta-2i}(\pi) \times \Lambda^{2i+1}(\pi) \boxtimes \widetilde{\tau_{\beta+1}}) \geq 1
\end{align*}
for $i=0$. This implies that there exists an isobaric summand $\tau_{\beta}$ in $\Sym^{k-\gamma+\beta}(\pi)$ such that
\begin{align*}
    \tau_{\beta} \prec \widetilde{\pi} \boxtimes \tau_{\beta+1} .
\end{align*}
Since $\tau_{\beta}$ has degree at least $\delta_{n,k-\gamma+\beta}(\beta)$ by the inductive hypothesis, we have the inequality
\begin{align*}
    \delta_{n,k-\gamma+\beta}(\beta) \leq n r_{\beta+1},
\end{align*}
which gives
\begin{align*}
    r_{\beta+1} \geq \ceil{\frac{1}{n} \delta_{n,k-\gamma+\beta}(\beta)} = \delta_{n,k-\gamma+\beta+1}(\beta+1) . 
\end{align*}
To justify the final equality, we observe that for any $a \in \RR$ and for any $m \in \NN$,
\begin{align}\label{eqn:elementary_ceiling_equality}
    \ceil{\frac{\ceil{a}}{m}} = \ceil{\frac{a}{m}} .
\end{align}
If we set
\begin{align*}
    a = \frac{1}{n^{\beta}} \min_{0 \leq i \leq \floor{\frac{n-1}{2}}} \ceil{\binom{n+k-\gamma-1-2i}{n-1} \bigg/ \binom{n}{2i+1}} ,
\end{align*}
then
\begin{align*}
    \ceil{\frac{1}{n} \delta_{n,k-\gamma+\beta}(\beta)} = \ceil{\frac{\ceil{a}}{n}} \qquad \text{ and } \qquad \delta_{n,k-\gamma+\beta+1}(\beta+1) = \ceil{\frac{a}{n}} .
\end{align*}
The equality of these two expressions follows from~\eqref{eqn:elementary_ceiling_equality}.

\underline{Case II:}
Suppose that
\begin{align*}
    -\ord_{s=1} L^{T}(s,\Sym^{k-\gamma+\beta-2i}(\pi) \times \Lambda^{2i+1}(\pi) \boxtimes \widetilde{\tau_{\beta+1}}) \geq 1
\end{align*}
for some $1 \leq i \leq \floor{\frac{\beta-1}{2}}$. This implies that there exists an isobaric summand $\tau_{\beta-2i}$ in $\Sym^{k-\gamma+\beta-2i}(\pi)$ such that
\begin{align*}
    \tau_{\beta-2i} \prec \widetilde{\Lambda^{2i+1}(\pi)} \boxtimes \tau_{\beta+1} .
\end{align*}
Since $\tau_{\beta-2i}$ has degree at least $\delta_{n,k-\gamma+\beta-2i}(\beta-2i)$ by the inductive hypothesis, we have the inequality
\begin{align*}
    \delta_{n,k-\gamma+\beta-2i}(\beta-2i) \leq \binom{n}{2i+1} r_{\beta+1}.
\end{align*}
Therefore,
\begin{align*}
    r_{\beta+1} \geq \min_{1 \leq j \leq \floor{\frac{\beta-1}{2}}} \ceil{\frac{\delta_{n,k-\gamma+\beta-2j}(\beta-2j)}{\binom{n}{2j+1}}} \geq \delta_{n,k-\gamma+\beta+1}(\beta+1). 
\end{align*}
To verify the final inequality, we set
\begin{align*}
    b = \frac{1}{n^{\beta-2j-1}} \min_{0 \leq i \leq \floor{\frac{n-1}{2}}} \ceil{\binom{n+k-\gamma-1-2i}{n-1} \bigg/ \binom{n}{2i+1}} .
\end{align*}
Then, we have
\begin{align*}
    \ceil{\frac{\delta_{n,k-\gamma+\beta-2j}(\beta-2j)}{\binom{n}{2j+1}}} = \ceil{\frac{\ceil{b}}{\binom{n}{2j+1}}} \qquad \text{ and } \qquad \delta_{n,k-\gamma+\beta+1}(\beta+1) = \ceil{\frac{b}{n^{2j+1}}} .
\end{align*}
Since $\binom{n}{2j+1} \leq n^{2j+1}$, we conclude from~\eqref{eqn:elementary_ceiling_equality} that
\begin{align*}
    \ceil{\frac{\delta_{n,k-\gamma+\beta-2j}(\beta-2j)}{\binom{n}{2j+1}}} = \ceil{\frac{\ceil{b}}{\binom{n}{2j+1}}} \geq \ceil{\frac{\ceil{b}}{n^{2j+1}}} = \ceil{\frac{b}{n^{2j+1}}} = \delta_{n,k-\gamma+\beta+1}(\beta+1) .
\end{align*}

\underline{Case III:}
Suppose that
\begin{align*}
    -\ord_{s=1} L^{T}(s,\Sym^{k-\gamma+\beta-2i}(\pi) \times \Lambda^{2i+1}(\pi) \boxtimes \widetilde{\tau_{\beta+1}}) \geq 1
\end{align*}
for some $\floor{\frac{\beta-1}{2}} < i \leq \floor{\frac{n-1}{2}}$. We know that $\Sym^{k-\gamma+\beta-2i}(\pi)$ is cuspidal in this case. Then
\begin{align*}
    \Sym^{k-\gamma+\beta-2i}(\pi) \prec \widetilde{\Lambda^{2i+1}(\pi)} \boxtimes \tau_{\beta+1} 
\end{align*}
and hence
\begin{align*}
    \binom{n+k-\gamma+\beta-1-2i}{n-1} \leq \binom{n}{2i+1} r_{\beta+1} .
\end{align*}
Therefore, 
\begin{align*}
    r_{\beta+1} \geq \min_{\floor{\frac{\beta-1}{2}} < i \leq \floor{\frac{n-1}{2}}} \ceil{\frac{\binom{n+k-\gamma+\beta-1-2i}{n-1}}{\binom{n}{2i+1}}} \geq \delta_{n,k-\gamma+\beta+1}(\beta+1) . 
\end{align*}

In all three cases, we have shown that $r_{\beta+1} \geq \delta_{n,k-\gamma+\beta+1}(\beta+1)$. Thus, by induction, the degree of every isobaric cuspidal summand of $\Sym^{k-\gamma+\alpha}(\pi)$ is at least $\delta_{n,k-\gamma+\alpha}(\alpha)$ for all $1 \leq \alpha \leq \gamma$. In particular, the degree of every isobaric cuspidal summand of $\Sym^{k}(\pi)$ is at least $\delta_{n,k}(\gamma)$.
\end{proof}

\begin{remark}
Let $k \geq n^{2+\varepsilon}$ for some $\varepsilon > 0$. The bound in Theorem~\ref{thm:Main_Theorem_var1} is trivial whenever $\gamma \geq \frac{n \log k}{\log n}$.
It suffices to prove that if $\gamma \geq \frac{n \log k}{\log n}$, then
\begin{align}\label{eqn:range_of_gamma_for_trivial}
    \frac{1}{n^{\gamma-1}}\ceil{\binom{n+k-\gamma-1-2\floor{\frac{n}{4}}}{n-1} \bigg/ \binom{n}{2\floor{\frac{n}{4}}+1}} \leq 1,
\end{align}
which implies that $\delta_{n,k}(\gamma) = 1$.

Applying Stirling's formula (see, e.g.~\cite{Robbins_on_Stirling_formula_1955}), we obtain
\begin{align*}
    &\log \binom{n+k-\gamma-1-2\floor{\frac{n}{4}}}{n-1} \\
    \leq& (n-1) \log (k+\frac{n}{2} - 1) - \log ((n-1)!) \\
    \leq& (n-1) \log k + (n-1) \log \left( 1+\frac{n-2}{2k} \right) - (n-1) \log (n-1) + n-1 - \frac{1}{2}\log(2\pi (n-1)) \\
    \leq& (n-1) \log k - (n-1) \log (n-1) + n-1 - \frac{1}{2}\log(n-1) .
\end{align*}
The last inequality holds because $(n-1) \log \left( 1+\frac{n-2}{2k} \right) \leq \frac{(n-1)(n-2)}{2k} \leq \frac{1}{2} \leq \frac{1}{2}\log(2\pi)$ when $k \geq n^{2+\varepsilon}$.

Next, we claim that for all $n \geq 2$, 
\begin{align*}
    \binom{n}{2\floor{\frac{n}{4}}+1} > \frac{2^{n-1}}{n+1} .
\end{align*}
If $n \equiv 1,2,3 \pmod{4}$, then
\begin{align*}
    \binom{n}{2\floor{\frac{n}{4}}+1} = \binom{n}{\floor{\frac{n}{2}}} \geq \binom{2 \floor{\frac{n}{2}}}{\floor{\frac{n}{2}}} > \frac{4^{\floor{\frac{n}{2}}}}{2\floor{\frac{n}{2}}+1} \geq \frac{4^{\frac{n-1}{2}}}{2\frac{n}{2}+1} = \frac{2^{n-1}}{n+1} ,
\end{align*}
by the estimate for the central binomial coefficient $\binom{2r}{r} > \frac{4^{r}}{2r+1}$~\cite{Wei_central_binomial_2022}.
On the other hand, for $n \equiv 0 \pmod{4}$, we have
\begin{align*}
    \binom{n}{2\floor{\frac{n}{4}}+1} = \binom{n}{\frac{n}{2}} \cdot \frac{\frac{n}{2}}{\frac{n}{2} + 1} > \frac{2^{n}}{n+1} \cdot \frac{n}{n+2} > \frac{2^{n-1}}{n+1} .
\end{align*}

We claim that 
\begin{align*}
    &\log \binom{n+k-\gamma-1-2\floor{\frac{n}{4}}}{n-1} - \log \binom{n}{2\floor{\frac{n}{4}}+1} \leq (n-1) \log k + \log n .
\end{align*}
For $n=2$, this is trivial.
For $n \geq 3$, we combine the estimates above to obtain
\begin{align*}
    &\log \binom{n+k-\gamma-1-2\floor{\frac{n}{4}}}{n-1} - \log \binom{n}{2\floor{\frac{n}{4}}+1} \\
    \leq & (n-1) \log k  - (n-1) \log (n-1) + n-1 - \frac{1}{2}\log (n-1) - (n-1)\log 2 + \log (n+1) .
\end{align*}
The claim now follows from the observation that 
\begin{align*}
    - (n-1) \log (n-1) + n-1 - \frac{1}{2}\log (n-1) - (n-1)\log 2 + \log (n+1) - \log n \leq 0
\end{align*}
for all $n \geq 3$. Indeed, the derivative of the left-hand side is negative for all $n \geq 2$, so the function is decreasing. Evaluating at $n=3$ verifies the claim.

Finally, if $\gamma \geq \frac{n \log k}{\log n}$, then 
\begin{align*}
    &\log \binom{n+k-\gamma-1-2\floor{\frac{n}{4}}}{n-1} - \log \binom{n}{2\floor{\frac{n}{4}}+1} - (\gamma - 1) \log n \\
    \leq & (n-1) \log k + \log n - \left(\frac{n \log k}{\log n} -1 \right) \log n \\
    \leq &-\log k + 2 \log n < 0 ,
\end{align*}
since $k \geq n^{2+\varepsilon}$ by assumption. It follows that~\eqref{eqn:range_of_gamma_for_trivial} holds, and consequently the bound in Theorem~\ref{thm:Main_Theorem_var1} becomes trivial.
\end{remark}

\begin{remark}
Let $k \geq n^{2+\varepsilon}$ for some $\varepsilon > 0$ and $\gamma < n \frac{\log k}{\log n}$.
We show, in a similar fashion to Remark~\ref{rmk:min_location_for_main_thm}, that the minimum of
\begin{align}\label{eqn:expression_of_delta_nkgamma}
    \ceil{\frac{\binom{n+k-\gamma-1-2i}{n-1}}{\binom{n}{2i+1}}}
\end{align}
over $0 \le i \le \floor{\frac{n-1}{2}}$, appearing in the definition of $\delta_{n,k}(\gamma)$ in~\eqref{eqn:delta_nk_gamma_def}, is attained at $i=\floor{\frac{n}{4}}$.
The numerator decreases in $i$, whereas the denominator increases for $0 \leq i \leq \floor{\frac{n}{4}}$ and decreases thereafter. Hence, it suffices to consider only the range $\floor{\frac{n}{4}} \leq i \leq \floor{\frac{n-1}{2}}$.

For $n=2,4$, we have $\floor{\frac{n}{4}} = \floor{\frac{n-1}{2}}$. The claim is immediate. 
For $n = 3$ or $n \geq 5$, consider the ratio 
\begin{align}\label{eqn:ratio_for_remark_3.3.1}
\begin{aligned}
    &\frac{\binom{n+k-\gamma-1-2(i+1)}{n-1}}{\binom{n}{2(i+1)+1}} \bigg/ \frac{\binom{n+k-\gamma-1-2i}{n-1}}{\binom{n}{2i+1}} \\
    =& \frac{(k-\gamma-2i)(k-\gamma-1-2i)}{(n+k-\gamma-1-2i)(n+k-\gamma-2-2i)} \cdot \frac{(2i+3)(2i+2)}{(n-2i-1)(n-2i-2)} 
\end{aligned}
\end{align}
where $\floor{\frac{n}{4}} \leq i \leq \floor{\frac{n-1}{2}}-1$.
The first factor $\frac{(k-\gamma-2i)(k-\gamma-1-2i)}{(n+k-\gamma-1-2i)(n+k-\gamma-2-2i)}$ is a decreasing function of $i$ and satisfies
\begin{align*}
    \frac{(k-\gamma-2i)(k-\gamma-1-2i)}{(n+k-\gamma-1-2i)(n+k-\gamma-2-2i)} 
    \geq \frac{(k-n-\gamma+3)(k-n-\gamma+2)}{(k-\gamma+2)(k-\gamma+1)} 
    > \frac{(k-n-\gamma+2)^{2}}{(k-\gamma+1)^{2}}.
\end{align*}
Recall the bound for $\frac{(2i+3)(2i+2)}{(n-2i-1)(n-2i-2)}$ in~\eqref{eqn:binomial_ratio_bounds}. It suffices to show that the ratio in~\eqref{eqn:ratio_for_remark_3.3.1} is greater than 1. We now show that
\begin{align*}
    \left(\frac{k-n-\gamma+2}{k-\gamma+1}\right) \left( \frac{n+2}{n-2} \right)> 1 ,
\end{align*}
or equivalently that $4(k-\gamma)> n^2+n-6$ by rearranging terms.
Write $k = n^{r}$ with $r \geq 2+\varepsilon$ for some $\varepsilon > 0$. Then $\gamma < n \frac{\log k}{\log n} = rn$. It therefore suffices to show that
\begin{align*}
    4(k-\gamma) - n^2 - n + 6 > 4(n^{r} - rn) - n^2 - n + 6 = 4n^{r} - n^{2} - (4r+1)n + 6
\end{align*}
is positive. Differentiating the right-hand side with respect to $r$ gives $4n^{r} \log n - 4n$, which shows that the function is increasing in $r$. Hence, for $r \geq 2+\varepsilon > 2$ and $n \geq 3$, we have 
\begin{align*}
    4n^{r} - n^{2} - (4r+1)n + 6 > 3n^{2} - 9n + 6 = 3(n-1)(n-2) > 0,
\end{align*}
proving the claim.
\end{remark}

\begin{remark}
Let $k \geq 2n^{2+\varepsilon}$ and $\gamma < \frac{n \log k}{\log n}$ for some $\varepsilon > 0$. 
Similar to Remark~\ref{rmk:asymptotics_for_main_thm}, we study the asymptotics of the bound
\begin{equation*}
    \mathcal{N}(\Sym^{k}(\pi)) \leq \floor{ \binom{n+k-1}{n-1} \bigg/ \frac{1}{n^{\gamma-1}} \ceil{\frac{\binom{n+k-\gamma-1-2\floor{\frac{n}{4}}}{n-1}}{\binom{n}{2\floor{\frac{n}{4}}+1}}} } .
\end{equation*}
It suffices to analyze 
\begin{equation*}
    n^{\gamma-1} \cdot \binom{n+k-1}{n-1} \bigg/ \ceil{\frac{\binom{n+k-\gamma-1-2\floor{\frac{n}{4}}}{n-1}}{\binom{n}{2\floor{\frac{n}{4}}+1}}} .
\end{equation*}
We first consider the ratio
\begin{align*}
    \frac{\binom{n+1+k-\gamma-1-2\floor{\frac{n+1}{4}}}{n+1-1}}{\binom{n+1}{2\floor{\frac{n+1}{4}}+1}} \bigg/ \frac{\binom{n+k-\gamma-1-2\floor{\frac{n}{4}}}{n-1}}{\binom{n}{2\floor{\frac{n}{4}}+1}} = \frac{\binom{n+k-\gamma-2\floor{\frac{n+1}{4}}}{n}}{\binom{n+k-\gamma-1-2\floor{\frac{n}{4}}}{n-1}} \cdot \frac{\binom{n}{2\floor{\frac{n}{4}}+1}}{\binom{n+1}{2\floor{\frac{n+1}{4}}+1}} .
\end{align*}
We write $k = n^{r}$ for some $r \geq 2 + \varepsilon$ so that $\gamma \leq n \frac{\log k}{\log n} = rn$. For the first factor, we see that
\begin{align*}
    \frac{\binom{n+k-\gamma-2\floor{\frac{n+1}{4}}}{n}}{\binom{n+k-\gamma-1-2\floor{\frac{n}{4}}}{n-1}} 
    &\geq \frac{\binom{n+k-\gamma-2\floor{\frac{n}{4}}-2}{n}}{\binom{n+k-\gamma-1-2\floor{\frac{n}{4}}}{n-1}}
    = \frac{(k-2\floor{\frac{n}{4}}-\gamma)(k-2\floor{\frac{n}{4}}-\gamma-1)}{n(k+n-2\floor{\frac{n}{4}}-\gamma-1)} \\
    &\geq \frac{(k-(r+\frac{1}{2})n)(k-(r+\frac{1}{2})n-1)}{n(k+\frac{n}{2}-\frac{5}{2})} .
\end{align*}
Applying~\eqref{eqn:second_factor_analysis_012} and~\eqref{eqn:second_factor_analysis_3} to the second factor, we see that the ratio 
\begin{align*}
    \frac{\binom{n+1+k-\gamma-1-2\floor{\frac{n+1}{4}}}{n+1-1}}{\binom{n+1}{2\floor{\frac{n+1}{4}}+1}} \bigg/ \frac{\binom{n+k-\gamma-1-2\floor{\frac{n}{4}}}{n-1}}{\binom{n}{2\floor{\frac{n}{4}}+1}} 
    \geq \frac{(k-(r+\frac{1}{2})n)(k-(r+\frac{1}{2})n-1)}{n(k+\frac{n}{2}-\frac{5}{2})} \cdot \frac{n-2\floor{\frac{n}{4}}}{n+1} \gg \frac{k}{n} \to \infty
\end{align*}
as $n \to \infty$, since $k \geq n^{2+\varepsilon}$.

We have shown that 
\begin{align*}
    \frac{\binom{n+k-\gamma-1-2\floor{\frac{n}{4}}}{n-1}}{\binom{n}{2\floor{\frac{n}{4}}+1}} \to \infty.
\end{align*}
Note that if $g(x) \to \infty$ as $x \to \infty$, then $\frac{1}{\ceil{g(x)}} \sim \frac{1}{g(x)}$.
Hence,  
\begin{align*}
    \binom{n+k-1}{n-1} \bigg/ \ceil{\frac{\binom{n+k-\gamma-1-2\floor{\frac{n}{4}}}{n-1}}{\binom{n}{2\floor{\frac{n}{4}}+1}}} \sim \frac{\binom{n+k-1}{n-1}}{\binom{n+k-\gamma-1-2\floor{\frac{n}{4}}}{n-1}} \cdot \binom{n}{2\floor{\frac{n}{4}}+1} .
\end{align*}
Recall that $k = n^{r}$ and $\gamma \leq rn$. Observe that
\begin{align*}
    \frac{\binom{n+k-1}{n-1}}{\binom{n+k-\gamma-1-2\floor{\frac{n}{4}}}{n-1}}
    = \prod_{j=0}^{n-2} \left( 1+ \frac{\gamma + 2\floor{\frac{n}{4}}}{n+k-\gamma-1-2\floor{\frac{n}{4}}-j}\right)
    \leq \left( 1+ \frac{(r+\frac{1}{2})n}{n^{r}-(r-\frac{1}{2})n}\right)^{n-1}.
\end{align*}
Since $r - 1 \geq 1+\varepsilon$ for some $\varepsilon >0$, the right-hand side tends to $1$ as $n \to \infty$. Therefore,
\begin{align*}
    \binom{n+k-1}{n-1} \bigg/ \ceil{\frac{\binom{n+k-\gamma-1-2\floor{\frac{n}{4}}}{n-1}}{\binom{n}{2\floor{\frac{n}{4}}+1}}} \sim \binom{n}{2\floor{\frac{n}{4}}+1} \sim \frac{2^{n+\frac{1}{2}}}{\sqrt{\pi n}} ,
\end{align*}
which is independent of the specific value of $k$. We refer the reader to Remark~\ref{rmk:asymptotics_for_main_thm} for the last asymptotic equivalence.
\end{remark}

\section{Bounding isobaric summands of $\Sym^{2}(\pi)$}
Walji~\cite[Proposition 3.7]{Walji_conjectural_decomposition_2022} proves an upper bound for $\mathcal{N}(\Sym^{2}(\pi))$, where $\pi \in \mathcal{A}_{0}(\GL_{3}(\mathbb{A}_{F}))$. 
We generalize this argument to all $\pi \in \mathcal{A}_{0}(\GL_{p^{n}}(\mathbb{A}_{F}))$, where $p$ is a prime and $n$ is a positive integer.
\begin{thm}\label{thm:bounding_iso_sum_for_sym2}
Let $p$ be a prime and $n$ be a positive integer. Let $\pi \in \mathcal{A}_{0}(\GL_{p^{n}}(\mathbb{A}_{F}))$. Suppose that $\Sym^{2}(\pi)$, $\Lambda^{2}(\pi)$ and $\pi \times \widetilde{\pi}$ are automorphic, then exactly one of the following holds:
\begin{enumerate}[(a)]
    \item Both $\Sym^{2}(\pi)$ and $\Lambda^{2}(\pi)$ are isobaric sums consisting solely of Hecke characters; or
    \item 
    \begin{align*}
        \mathcal{N}(\Sym^{2}(\pi)) \leq \floor{\frac{p^{2n}+2p^{2n-1}+p^{n}}{4}} .
    \end{align*}
\end{enumerate}
\end{thm}

\begin{proof}
The idea is to bound the number of Hecke characters that are isobaric summands of $\Sym^{2}(\pi)$. Suppose there exists Hecke character $\omega$ that is an isobaric summand of $\Sym^{2}(\pi)$. Given that
\begin{align}\label{eqn:Sym2_Lambda2}
    L^{T}(s, \pi \times \pi \otimes \omega^{-1}) = L^{T}(s, \Sym^{2}(\pi) \otimes \omega^{-1}) L^{T}(s, \Lambda^{2}(\pi) \otimes \omega^{-1}),
\end{align}
the first $L$-function must have a pole at $s = 1$, and hence $\widetilde{\pi} \simeq \pi \otimes \omega^{-1}$. 

Let $S = \{ \omega \in \mathcal{A}(\GL_{1}(\mathbb{A}_{F})) : \pi \otimes \omega \simeq \widetilde{\pi} \}$. To bound $\abs{S}$, we let $G = \{ \mu \in \mathcal{A}(\GL_{1}(\mathbb{A}_{F})) : \pi \otimes \mu \simeq \pi \}$ be the multiplicative group of self-twists of $\pi$.
Since $G$ acts faithfully and transitively on the set $S$, they have the same cardinality. 
Note that if $\mu \in G$, then $\mu$ is an isobaric summand of $\pi \boxtimes \widetilde{\pi}$ and hence $\abs{G} \leq p^{2n}$. 
By comparing the central characters of $\pi \otimes \mu$ and $\pi$, we conclude that every element of $G$ has order dividing $p^{n}$. Hence, $G$ is a $p$-group, and $\abs{G} = p^{2n}$ or $\abs{G} \leq p^{2n-1}$. If $\abs{S} = \abs{G} = p^{2n}$, then this is case (a).
If $\abs{S} = \abs{G} \leq p^{2n-1}$, then~\eqref{eqn:Sym2_Lambda2} implies that the number of Hecke characters appearing as isobaric summands of $\Sym^{2}(\pi)$ is at most $p^{2n-1}$. All remaining isobaric summands of $\Sym^{2}(\pi)$ have degree at least $2$. Therefore,
\begin{align*}
    \mathcal{N}(\Sym^{2}(\pi)) \leq \floor{\frac{\frac{p^{2n}+p^{n}}{2}-p^{2n-1}}{2}} + p^{2n-1} = \floor{\frac{p^{2n}+2p^{2n-1}+p^{n}}{4}} .
\end{align*}
\end{proof}

\section{Examples}\label{sec:consequences}
\subsection{Sharpness of Theorem~\ref{thm:Main_Theorem_var1}}
In this section, we show that the bound in Theorem~\ref{thm:Main_Theorem_var1} is sharp in a specific $\GL(2)$ setting.
We begin by recalling the notion of a quasi-icosahedral representation, introduced by Ramakrishnan~\cite[p.2]{Ramakrishnan_sym_power_cusp_form_2009}.
\begin{definition}
Let $\pi \in \mathcal{A}_{0}(\GL_{2}(\mathbb{A}_{F}))$ be a cuspidal representation. We say that $\pi$ is quasi-icosahedral if 
\begin{enumerate}[(i)]
    \item $\Sym^{m}(\pi)$ is automorphic for all $m \leq 6$,
    \item $\Sym^{m}(\pi)$ is cuspidal for all $m \leq 4$, and
    \item $\Sym^{6}(\pi)$ is not cuspidal.
\end{enumerate}
\end{definition}

We can now apply Theorem~\ref{thm:Main_Theorem_var1} to study the decomposition of $\Sym^{7}(\pi)$, where $\pi$ is a quasi-icosahedral representation.
\begin{cor}
Let $\pi \in \mathcal{A}_{0}(\GL_{2}(\mathbb{A}_{F}))$ be quasi-icosahedral. Assume that $\Sym^{7} (\pi)$ is automorphic. If $\Sym^{7} (\pi)$ has a cuspidal isobaric summand $\eta$ for some $\GL(r)$, then $r \geq 2$. Moreover, this bound is sharp.
\end{cor}
\begin{proof}
Observe that $\Sym^{5}(\pi)$ is automorphic by the definition of $\pi$ being quasi-icosahedral. Moreover, $\Sym^{5}(\pi)$ is cuspidal~\cite[Theorem A]{Ramakrishnan_sym_power_cusp_form_2009}.
We now apply the lower bound for $r$, which is $\delta_{2,7}(2)$ in the proof of Theorem~\ref{thm:Main_Theorem_var1} and obtain
\begin{align}\label{eqn:bound_example}
    r \geq \delta_{2,7}(2) = \ceil{\frac{1}{2} \ceil{\binom{6}{1} \bigg/ \binom{2}{1} }} = 2 .
\end{align}

We now give an example to show that this bound is sharp.
We fix an Artin representation $\rho: \Gal(\overline{\QQ}/\QQ) \to \GL_{2}(\CC)$ which is of odd icosahedral type. Recall that $\rho$ is odd icosahedral if $\det \rho(c) = -1$, where $c$ is a complex conjugation and the projective image of $\rho$ is isomorphic to the icosahedral group $A_{5}$.
Wang~\cite[Proposition 2.1]{Wang_icosahedral_sym_2003} showed that $\rho$ factors through $(\widetilde{A_{5}} \times \mu_{2m})/ \{ \pm I\}$, where $\widetilde{A_{5}}$ is the binary icosahedral group and $\mu_{2m}$ is the group of roots of unity of order $2m$.
Hence, $\rho$ decomposes into a pair of representations $(\Lambda_{0}, \chi)$, where $\Lambda_{0}$ is a representation for $\widetilde{A_{5}}$ and $\chi$ is a representation for $\mu_{2m}$. 

Here is the character table for $\widetilde{A_{5}}$, which includes only the two complex irreducible representations of degree $2$.
\begin{align*}
\begin{array}{c|c|c|c|c|c|c|c|c|c}
     & [1] & [-1] & [\alpha] & [\alpha^{2}] & [\alpha^{3}] & [\alpha^{4}] & [\beta] & [\beta^{2}] & [\gamma] \\
    \hline
    \rho_{ico} & 2 & -2 & \frac{1+\sqrt{5}}{2} & -\frac{1-\sqrt{5}}{2} & \frac{1-\sqrt{5}}{2} & -\frac{1+\sqrt{5}}{2} & 1 & -1 & 0 \\
    \rho_{ico}^{\tau} & 2 & -2 & \frac{1-\sqrt{5}}{2} & -\frac{1+\sqrt{5}}{2} & \frac{1+\sqrt{5}}{2} & -\frac{1-\sqrt{5}}{2} & 1 & -1 & 0 
\end{array}
\end{align*}
where $\tau$ is the non-trivial element in $\Gal(\QQ(\sqrt{5})/\QQ)$ and $\rho_{ico}^{\tau} = \rho_{ico} \circ \tau$.
Since $\Lambda_{0}$ is irreducible and of degree $2$, it must be isomorphic to $\rho_{ico}$ or $\rho_{ico}^{\tau}$.

Without loss of generality, we assume that $\Lambda_{0}$ is isomorphic to $\rho_{ico}$. Consider also the odd icosahedral representation $\rho^{\tau}$, which decomposes into $(\rho_{ico}^{\tau}, \chi)$. By~\cite[Corollary 1.2]{Wang_icosahedral_sym_2003}, we have $\Sym^{3}(\rho) = \Sym^{3}(\rho^{\tau})$. 
The strong Artin conjecture, which is known for the odd icosahedral case due to the work of Khare-Wintenberger~\cite{Khare_Wintenberger_modularity_conjecture_I_2009,Khare_Wintenberger_modularity_conjecture_II_2009} and Kisin~\cite{Kisin_modularity_2009}, implies that there exists $\pi, \pi^{\prime} \in \mathcal{A}_{0}(\GL_{2}(\mathbb{A}_{F}))$ such that 
\begin{align*}
    L(s, \pi) = L(s, \rho) \qquad \text{ and } \qquad  L(s, \pi^{\prime}) = L(s, \rho^{\tau}) .
\end{align*}

Observe that $\Sym^{3}(\pi) = \Sym^{3}(\pi^{\prime})$. Ramakrishnan~\cite[Corollary B]{Ramakrishnan_recovering_cusp_form_2015} showed that 
\begin{align*}
    \Sym^{7}(\pi) \simeq \Sym^{2}(\pi) \boxtimes \pi^{\prime} \otimes \omega^{2} \boxplus \pi^{\prime} \otimes \omega^{3}
\end{align*}
where $\omega$ is the central character associated to $\pi$. Since $\pi^{\prime}$ is of degree $2$, we have constructed an example demonstrating that the bound~\eqref{eqn:bound_example} is sharp.
\end{proof}

\subsection{Sharpness of Theorem~\ref{thm:bounding_iso_sum_for_sym2}}
The bound in case (b) of Theorem~\ref{thm:bounding_iso_sum_for_sym2} is sharp, as shown by the three-dimensional irreducible representation $\rho$ of the group $A_{4}$~\cite[Section 4.1.5]{Walji_conjectural_decomposition_2022}. Its symmetric square decomposes into three $1$-dimensional representations and one $3$-dimensional representation, and hence 
\begin{align*}
    \mathcal{N}(\Sym^{2}(\rho)) = 4 = \floor{\frac{3^{2}+2\cdot 3+3}{4}} .
\end{align*}
By the strong Artin conjecture for solvable extensions, there exists a cuspidal automorphic representation $\pi$ for $\GL(3)$ corresponding to $\rho$, and hence $\mathcal{N}(\Sym^{2}(\pi)) = 4$.

We note that case (a) in Theorem~\ref{thm:bounding_iso_sum_for_sym2} can also occur. Consider the 2-dimensional irreducible representation $\sigma$ of the dihedral group $D_{8}$ of order $8$. It is known that $\sigma \otimes \sigma \simeq \Sym^{2}(\sigma) \oplus \Lambda^{2}(\sigma)$, as a representation of $D_{8}$, can be decomposed into four 1-dimensional representations. 
By the strong Artin conjecture, there exists a cuspidal automorphic representation $\pi$ for $\GL(2)$ such that both $\Sym^{2}(\pi)$ and $\Lambda^{2}(\pi)$ are isobaric sums of Hecke characters.
We can extend this example to a setting in $\GL(2^{n})$.
Let $D_{8}^{n}$ be the direct product of $n$ copies of $D_{8}$, and set $\sigma^{\prime} = \bigotimes_{i=1}^{n} \sigma$. Then $\sigma^{\prime}$ is an irreducible representation of $D_{8}^{n}$ of dimension $2^{n}$. 
Using the distributivity of the tensor product over direct sums
\begin{align*}
    (\sigma_{1} \oplus \sigma_{2}) \otimes (\sigma_{3} \oplus \sigma_{4}) \simeq (\sigma_{1} \otimes \sigma_{3}) \oplus (\sigma_{1} \otimes \sigma_{4}) \oplus (\sigma_{2} \otimes \sigma_{3}) \oplus (\sigma_{2} \otimes \sigma_{4}) ,
\end{align*}
together with the decomposition of $\sigma \otimes \sigma$ into 1-dimensional representations, it follows inductively that $\Sym^{2}(\sigma^{\prime})$ and $\Lambda^{2}(\sigma^{\prime})$ decompose as direct sums of 1-dimensional representations.

More generally, let $n = p_{1}^{r_{1}} \cdots p_{k}^{r_{k}}$ be its prime factorization, and let $m = p_{1} \cdots p_{k}$ denote the product of the distinct primes dividing $n$.
For any $k$ divisible by $m$, there exists a cuspidal automorphic representation $\pi$ for $\GL(n)$ such that $\pi^{\otimes k}$ is an isobaric sum of Hecke characters. The same holds for $\Sym^{k} (\pi)$.

It suffices to show that $\pi^{\otimes m}$ is an isobaric sum of Hecke characters. Fix a prime $p$ dividing $n$. Consider the Heisenberg group $\He_{p}$, which is the semidirect product $(\ZZ/p\ZZ)^{2} \rtimes \ZZ/p\ZZ$. It has center $Z \cong \ZZ/p\ZZ$. Note that $\He_{2}$ is the dihedral group $D_{8}$.
The irreducible complex representations of $\He_{p}$ are given as follows (see e.g.~\cite{Misaghian_Heisenberg_group_repn_2010}):
\begin{itemize}
    \item There are exactly $p^{2}$ irreducible representations of dimension $1$. Their character values are 1 on the center $Z$.
    \item There are exactly $p-1$ irreducible representations of dimension $p$. Each corresponds to a non-trivial character $\chi'$ of the center $Z$. Their character values are $p \cdot \chi'$ on $Z$ and 0 for elements not in $Z$.
\end{itemize}
Let $\rho_{p}$ be a $p$-dimensional irreducible representation of $\He_{p}$. Then $\rho_{p}^{\otimes p}$ has dimension $p^{p}$. Its character values are $p^{p}$ on the center $Z$ and $0$ for elements not in $Z$. 
Let $\chi$ be a one-dimensional irreducible representation of $\He_{p}$. We compute the multiplicity of $\chi$ in $\rho_{p}^{\otimes p}$ by 
\begin{align*}
    \langle \rho_{p}^{\otimes p}, \chi \rangle = \frac{1}{\abs{G}} \sum_{g \in G} \rho_{p}^{\otimes p}(g) \overline{\chi(g)} = \frac{1}{p^{3}} \sum_{g \in Z} p^{p} = p^{p-2}.
\end{align*}
This shows that $\rho_{p}^{\otimes p}$ can be decomposed into a sum of $p^{2}$ distinct irreducible representations of dimension 1, each occurring with multiplicity $p^{p-2}$.

To construct an irreducible representation of dimension $n$ that decomposes as a sum of 1-dimensional representations, we follow the same idea as in the construction of the representation of $Q_{8}^{n}$ as above.
Consider $G = \prod_{i=1}^{k} \He_{p_{i}}^{r_{i}}$, the direct product of $r_{i}$ copies of $\He_{p_{i}}$. Then $\bigotimes_{i=1}^{k} \rho_{p_{i}}^{\otimes r_{i}}$ is an irreducible representation of $G$ of dimension $n$.
Hence, there exists a cuspidal automorphic representation $\pi$ for $\GL(n)$ such that $\pi^{\otimes m}$ is an isobaric sum of Hecke characters.

\bibliographystyle{alpha}
\bibliography{bibliography}

\end{document}